\documentclass[12pt]{amsart}
\usepackage{amscd}
\def\NZQ{\Bbb}               
\def\NN{{\NZQ N}}

\def\ZZ{{\NZQ Z}}

\def\frk{\frak}               

\def\Phi{{\frk n}}
\def\Phi{{\frk N}}
\def\opn#1#2{\def#1{\operatorname{#2}}} 
\opn\chara{char} \opn\length{\ell} \opn\pd{pd} \opn\rk{rk}
\opn\projdim{proj\,dim} \opn\injdim{inj\,dim} \opn\rank{rank}
\opn\depth{depth} \opn\grade{grade} \opn\height{height}
\opn\embdim{emb\,dim} \opn\codim{codim}

\opn\Tr{Tr} \opn\bigrank{big\,rank}
\opn\superheight{superheight}\opn\lcm{lcm}
\opn\trdeg{tr\,deg}
\opn\reg{reg} \opn\lreg{lreg} \opn\ini{in} \opn\lpd{lpd}
\opn\size{size} \opn\sdepth{sdepth} \opn\link{link}\opn\fdepth{fdepth}
\opn\div{div} \opn\Div{Div} \opn\cl{cl} \opn\Cl{Cl}
\opn\Spec{Spec} \opn\Supp{Supp} \opn\supp{supp} \opn\Sing{Sing}
\opn\Ass{Ass} \opn\Min{Min}
\opn\Ann{Ann} \opn\Rad{Rad} \opn\Soc{Soc}
\opn\Im{Im} \opn\Ker{Ker} \opn\Coker{Coker} \opn\Am{Am}
\opn\Hom{Hom} \opn\Tor{Tor} \opn\Ext{Ext} \opn\End{End}
\opn\Aut{Aut} \opn\id{id}

\opn\nat{nat}
\opn\pff{pf}
\opn\Pf{Pf} \opn\GL{GL} \opn\SL{SL} \opn\mod{mod} \opn\ord{ord}
\opn\Gin{Gin} \opn\Hilb{Hilb}
\opn\aff{aff} \opn\con{conv} \opn\relint{relint} \opn\st{st}
\opn\lk{lk} \opn\cn{cn} \opn\core{core} \opn\vol{vol}
\opn\link{link} \opn\star{star}
\opn\gr{gr}

\def\pot#1#2{#1[\kern-0.28ex[#2]\kern-0.28ex]}

\opn\dirlim{\underrightarrow{\lim}}
\opn\inivlim{\underleftarrow{\lim}}
\let\tensor=\otimes
\let\iso=\cong

\let\Dirsum=\bigoplus

\let\to=\rightarrow

\def\Implies{\ifmmode\Longrightarrow \else
        \unskip${}\Longrightarrow{}$\ignorespaces\fi}
\def\implies{\ifmmode\Rightarrow \else
        \unskip${}\Rightarrow{}$\ignorespaces\fi}
\def\iff{\ifmmode\Longleftrightarrow \else
        \unskip${}\Longleftrightarrow{}$\ignorespaces\fi}

\let\:=\colon
\newtheorem{Theorem}{Theorem}[section]
\newtheorem{Lemma}[Theorem]{Lemma}
\newtheorem{Corollary}[Theorem]{Corollary}
\newtheorem{Proposition}[Theorem]{Proposition}

\newtheorem{Example}[Theorem]{Example}

\let\epsilon\varepsilon
\let\phi=\varphi
\let\kappa=\varkappa
\def\qed{\ifhmode\textqed\fi
      \ifmmode\ifinner\quad\qedsymbol\else\dispqed\fi\fi}
\def\textqed{\unskip\nobreak\penalty50
       \hskip2em\hbox{}\nobreak\hfil\qedsymbol
       \parfillskip=0pt \finalhyphendemerits=0}
\def\dispqed{\rlap{\qquad\qedsymbol}}

\opn\dis{dis}
\def\pnt{{\raise0.5mm\hbox{\large\bf.}}}

\opn\Lex{Lex}

\begin{document}

\title{Stanley decompositions and localization}

\author{Sumiya Nasir}


\thanks{The author is highly grateful to the Abdus Salam School of Mathematical Sciences,
  GC University, Lahore, Pakistan in supporting and facilitating this
  research. The author would like to thank Prof. Jurgen Herzog for introducing the idea and encouragement.}

\address{Sumiya Nasir, School of Mathematical Sciences, 68-B New Muslim Town,
    Lahore,Pakistan.}\email{Sumiya.sms@gmail.com} \maketitle

\begin{abstract}
We study the behavior of Stanley depth under the operation of localization with respect to a variable.
\end{abstract}

\section*{Introduction}

 Let $K$ be a field, $S=K[x_1,\ldots,x_n]$ be the polynomial ring in
$n$ variables over $K$ and $I\subset S$ a monomial ideal. Stanley depth  of $S/I$ is denoted by $\sdepth S/I$, see Section 2 for its definition. The Stanley depth is an important
combinatorial invariant of $S/I$ studied in \cite{HSY}, \cite{HVZ}, \cite{R},
\cite{So}. The interest in this subject arises in part from the so-called Stanley conjecture which asserts that $\sdepth S/I\geq \depth S/I$.

The purpose of this note is to study the behavior of $\sdepth S/I$  under the operation of localization with  respect to a variable.
The effect of localization of a monomial ideal with respect to a variable, say $x_n$,  is, up to a flat extension, the same as applying the $K$-algebra homomorphism
$\varphi :S\rightarrow T=K[x_1,\ldots,x_{n-1}]$ given by $x_n\mapsto 1$. This is explained in Section 1.

Many, but not all, Stanley decompositions arise as prime filtrations. In Section~2  we  show how prime filtrations behave under localization, see Proposition~\ref{sumiya}. As a consequence we show in Corollary~\ref{nasir} that pretty clean filtrations induce under localization again pretty clean filtrations. This  implies in particular that if Stanley's conjecture holds for $S/I$, then it holds for the localization as well. As an immediate consequence of Proposition~\ref{sumiya} we show that $\fdepth T/\varphi(I)\geq \fdepth(S/I)-1$, where $\fdepth$, introduced in \cite{HVZ}, is an invariant  of $S/I$ related to prime filtrations. This invariant is of interest since one always has $\fdepth S/I \leq \sdepth S/I, \depth S/I$.

The main purpose of Section 3 is to prove an inequality analogue to that for the $\fdepth$. In fact, we show in Corollary~\ref{main} that $\sdepth T/\varphi(I)\geq \sdepth(S/I)-1$. Easy examples show that the inequality is often strict. On the other hand, we also give an example for which   $\sdepth
T/\varphi(I)> \sdepth(S/I)$.

When $I=I_{\Delta}$ is the
Stanley-Reisner ideal of a simplicial complex $\Delta$ we get in
particular that $\sdepth K[\link_{\Delta}(\{n\})]\geq \sdepth
K[\Delta]-1$, where $K[\Delta]=S/I$ (see Lemma \ref{link}).

\section{Localization of monomial ideals}

Let $K$ be a field and $S=K[x_1,\ldots, x_n]$ be the polynomial ring
in $n$ variables over $K$, and let $I\subset S$ a be a monomial
ideal. Suppose that $I$ is generated by the monomials $u_1,\ldots,
u_m$ with $u_i=\prod_{j^=1}^n x_j^{a_{ij}}$. We denote, as usual,
by $S_{x_n}$ the localization of $S$ with respect to the element
$x_n$. Notice that $S_{x_n}$ has a $K$-basis consisting of all
monomials of the form
\[
x_1^{a_1}x_2^{a_2}\cdots x_{n-1}^{a_{n-1}}x_n^{a_n}\quad
\text{with}\quad a_i\in\ZZ_{\geq 0}\quad \text{and}\quad  a_n\in
\ZZ.
\]
In other words,
\[
S_{x_n}=K[x_n,x_n^{-1}][x_1,\ldots,
x_{n-1}]=K[x_n,x_n^{-1}]\tensor_K T,
\]
where $T=  K[x_1,\ldots,x_{n-1}]$.

The extension ideal $IS_{x_n}$ is the ideal in $S_{x_n}$ which is
generated by the monomials $u_i'=\prod_{j^=1}^{n-1} x_j^{a_{ij}}$,
because the last variable becomes a unit.

Let $\varphi :\; S\to T$ be the $K$-algebra homomorphism with
$x_i\mapsto x_i$ for $i=1,\ldots ,n-1$ and $x_n\mapsto 1$, then
$\varphi(u_i)=u_i'$ for all $i$ and we see that $IS_{x_n}$ is the
extension ideal of $\varphi(I)$ under the flat extension $T\to
K[x_n,x_n^{-1}]\tensor_K T=S_{x_n}$.

\section{Localization of prime filtrations}

Let $K$ be a field and $S=K[x_1,\ldots, x_n]$ be the polynomial ring
in $n$ variables over $K$. Let $I\subset S$ be a monomial ideal. A
{\em prime filtration} of $S/I$ is a chain of monomial ideals
\[
{\mathcal P}: I=I_0\subset I_1\subset I_2\subset \cdots \subset
I_r=S
\]
such that  there are isomorphisms of $\ZZ^n$-graded $S$-modules
\[
I_j/I_{j-1}\iso (S/P_j)(-a_j) \quad \text{for} \quad j=1,2,\ldots,
r,
\]
where $P_j$  is a monomial prime ideal and $a_j\in \ZZ^n$. The set
$\{P_1,\ldots,P_r\}$  is called the {\em support} of $\mathcal P$
and denoted $\Supp({\mathcal P})$.

We consider the $K$-algebra homomorphism $\varphi\: S\to
T=K[x_1,\ldots,x_{n-1}]$, introduced in the previous section,  with
$x_i\mapsto x_i$ for $i=1,\ldots, n-1$ and $x_n\mapsto 1$. We will
also consider the projection map $\pi\:  \ZZ^n\to \ZZ^{n-1}$ which
assigns to each $a=(a_1,\ldots,a_n)$ in $\ZZ^n$ the vector
$a'=\pi(a)=(a_1,\ldots,a_{n-1})$.

\begin{Proposition}
\label{sumiya} Let $I\subset S$ be a monomial ideal, and let
${\mathcal P}$ be a prime filtration of $S/I$ as above. We set
$J=\varphi(I)$ and $J_j=\varphi(I_j)$ for all $I_j$ in the prime
filtration. Then we get the filtration
\[
J=J_0\subseteq J_1\subseteq J_2\subseteq \cdots \subseteq J_r=T
\]
with
\[
J_j/J_{j-1} \iso \left\{ \begin{array}{ll}
       (T/P'_j)(-a_j'), & \;\text{if  $x_n\not\in P_j$}, \\ 0, & \;\text{if $x_n\in P_j,$}
        \end{array} \right.
\]
where $P_j'\subset T$ is the monomial prime ideal in $T$ such that
$P_j=P_j'S$.
\end{Proposition}

\begin{proof}
The statement of the proposition follows once we can show  the
following: Let $I\subset J$ be monomial ideals in $S$ such that
$J/I\iso (S/P)(-a)$ where $P$ is a monomial prime ideal and $a\in
\ZZ_{\geq 0}^n$.  Then
\[
\varphi(J)/\varphi(I) \iso \left\{ \begin{array}{ll}
       (T/P')(-a'), & \;\text{if  $x_n\not\in P$}, \\
       0, & \;\text{if $x_n\in P$},
        \end{array} \right.
\]
\noindent We have   $J/I\iso (S/P)(-a)$ if and only if  $J=(I,x^a)$
and $I :_S x^a=P$. Since
\[
\varphi(J)=\varphi(I,x^a)=(\varphi(I), x^{a'}),
\]
we see that
\begin{eqnarray}
\label{madiha} \varphi(J)/\varphi(I)\iso (\varphi(I),
x^{a'})/\varphi(I))\iso (T/(\varphi(I) :_Tx^{a'}))(-a').
\end{eqnarray}
Next we claim that $\varphi(I :_S x^a)=(\varphi(I) :_T x^{a'})$.
Suppose this is true, then we get
\[
(\varphi(I) :_T x^{a'})=\varphi(P)= \left\{ \begin{array}{ll}
       P', & \;\text{if  $x_n\not\in P$}, \\
       T, & \;\text{if $x_n\in P$},
        \end{array} \right.
\]
Hence  the desired result follows.

It remains to prove the claim: let $I=(u_1,\ldots, u_m)$ with
$u_i=x^{a_i}=\prod_{j=1}^n x_j^{a_{ij}}$. Then
\begin{eqnarray*}
I :_Sx^a&=&(x^{a_1}/\gcd(x^{a_1},x^a),\ldots, x^{a_m}/\gcd(x^{a_m},x^a))\\
&=&(\prod_{j=1}^n x_j^{a_{1j}-\min\{a_{1j},a_j\}},\cdots,
\prod_{j=1}^n x_j^{a_{mj}-\min\{a_{mj},a_j\}}).
\end{eqnarray*}
It follows that
\begin{eqnarray*}
\varphi(I :_Sx^a)&=& (\varphi(\prod_{j=1}^n x_j^{a_{1j}-\min\{a_{1j},a_j\}}),\cdots, \varphi(\prod_{j=1}^n x_j^{a_{mj}-\min\{a_{mj},a_j\}}))\\
&=&(\prod_{j=1}^{n-1} x_j^{a_{1j}-\min\{a_{1j},a_j\}},\cdots, \prod_{j=1}^{n-1} x_j^{a_{mj}-\min\{a_{mj},a_j\}})\\
&=&(x^{a_1'}/\gcd(x^{a_1'},x^{a'}),\ldots, x^{a'_m}/\gcd(x^{a_m'},x^{a'}))\\
&=&(\varphi(x^{a_1})/\gcd(\varphi(x^{a_1}),\varphi(x^a)),\ldots, \varphi(x^a)/\gcd(\varphi(x^{a_m}),\varphi(x^a)))\\
&=&\varphi(I):_Tx^{a'}.
\end{eqnarray*}
\end{proof}

Let $K$ be a field and $S=K[x_1,\ldots,x_n]$ be a polynomial ring.
Let $I\subset S$ be a monomial ideal. A prime filtration
\[
  \mathcal {P}:I=I_0\subset I_1\subset\cdots \subset I_r=S
\] of
$S/I$ such that $I_j/I_{j-1}\iso(S/P_j)(-a_j)$ is said to be {\em
clean} (see \cite{Dr}) if $\Supp(\mathcal P)=\Min(S/I)$, where
$\Min(S/I)$ denotes the set of minimal prime ideals of $I$.
Equivalently, $(\mathcal P)$ is clean, if there is no containment
between the elements in $\Supp(\mathcal P)$, see  \cite{HP}. A
monomial ideal $I$ is said to be {\em clean} if $S/I$ has a clean
filtration. The prime filtration ${\mathcal P}$ is said to be {\em
pretty clean} if for all $i<j$ the inclusion  $P_i \subset P_j$
implies
 $P_i=P_j$ (see \cite{HP}). A monomial ideal $I$ is said to be
  {\em pretty clean} if $S/I$ has a pretty clean filtration.

Let $I\subset S$ be a monomial ideal. We denote by $I^{c}\subset S$
the $K$  linear subspace of $S$ generated by all monomials which do
not belong to $I$. Then  $S=I \oplus I^c$   and $S/I \iso I^{c}$ as
$K$-linear spaces.
 If $u\in S$ is a monomial and $Z \subset\{x_1,\ldots, x_n\}$, the
 $K$-subspace $ uK[Z]$ whose basis consists of all monomials $uv$
 with $v\in K[Z] $ is called a {\em Stanley space } of dimension $|Z |$.
 A decomposition $ \mathcal D$ of $I^c$ as a finite direct sum of {\em Stanley
 spaces} is called a {\em Stanley decomposition } of $S/I$. The minimal
 dimension of a Stanley spaces in $\mathcal D$ is called
 the {\em Stanley depth}  of ${\mathcal D}$ and is denoted by $\sdepth {\mathcal D}$.
Finally we define $\sdepth S/I$ by
\[
  \sdepth S/I= \max\{\sdepth \mathcal D  :\;  \mathcal  D \quad \text{is
  a
Stanley  decomposition  of  $S/I$}\}
\]
In \cite{St} Stanley conjectures that for any monomial ideal
$I\subset S$  one has $\sdepth S/I\geq \depth S/I$. The  monomial
ideal $I$ is said to be a {\em Stanley ideal} if Stanley's
conjecture holds for $S/I$.  It is shown in \cite{HP} that a pretty
clean ideal is a Stanley ideal.

As a consequence of the previous result we have

\begin{Corollary}
\label{nasir} Let  $I\subset S$  be a monomial ideal. If $I$ is
(pretty) clean, then $\varphi(I)\subset T$ is (pretty) clean. In
particular, if $I$ is pretty clean,  then  $\varphi(I)\subset T$ is
a Stanley ideal.
\end{Corollary}

\begin{proof}
We refer to the the hypotheses and notation  of
Proposition~\ref{sumiya}, and assume in addition that the filtration
$\mathcal P$ of $S/I$ is (pretty) clean. The filtration of $J$ given
in Proposition~\ref{sumiya} can be modified to give a prime
filtration of $T/J$ (by omitting for all $i>0$ those  $J_i$ for
which $J_{i-1}=J_i$) whose support is a subset of $\Supp({\mathcal
P})$. From this, all assertions follow immediately.
\end{proof}

Let ${\mathcal F}\: I=I_0\subset I_1\subset I_2\subset \cdots \subset I_r=S$ be a prime filtration  with $I_j/I_{j-1}\iso S/P_j(-a_j)$.  Then
\[
{\mathcal D(F)}: S/I=\Dirsum_{j=1}^ru_iK[Z_i]
\]
is a Stanley decomposition of $S/I$, where $u_i=x^{a_i}$ and $Z_i=
\{x_j :\; x_j\not\in P_i\}$ (see \cite{HP}). Thus if we set $\fdepth {\mathcal F}=\min\{\dim S/P_1,\ldots, \dim S/P_r\}$ and
\[
\fdepth S/I=\max\{\fdepth {\mathcal F}\:   \text{$\mathcal{F}$ is a prime filtration of $S/I$}\},
\]
then  see that
$\fdepth S/I\leq \sdepth S/I$.

As an immediate consequence of Proposition~\ref{sumiya} we obtain

\begin{Corollary}
\label{shehnaz} Let $I\subset S$ be a pretty clean monomial ideal.
Then
\[
\fdepth T/\varphi(I)\geq \fdepth S/I-1.
\]
\end{Corollary}

\section{Localizations and Stanley decompositions}

The purpose of this section is to prove an inequality for the $\sdepth$  similar to that for the $\fdepth$ given in Corollary~\ref{shehnaz} in Section 2. The desired inequality will be a consequence of

\begin{Theorem}
Let ${\mathcal D}:S/I=\bigoplus_{i=1}^ru_iK[Z_i]$ be a Stanley
decomposition of $S/I$ then ${\mathcal
D}':T/\varphi(I)=\bigoplus\limits_{x_n\in
Z_i}\varphi(u_i)K[Z_i\setminus \{x_n\}]$ is a Stanley decomposition
of $T/ \varphi(I).$
\end{Theorem}
\begin{proof}
Firstly we prove that
 $$ \varphi(u_i)K[Z_i\setminus \{x_n\}]\cap \varphi(u_j)K[Z_j\setminus \{x_n\}]=\{0\} $$
 for $ i\neq j $ and $ x_n \in Z_i ,Z_j$. Suppose on the contrary that there exists a monomial
 $u\in T$ such that $$ u \in \varphi(u_i)K[Z_i\setminus \{x_n\}]\cap \varphi(u_j)K[Z_j\setminus \{x_n\}], $$
 that is $$ u=\varphi(u_i)f_i=\varphi(u_j)f_j,$$ for some monomials
 $ f_i\in K[Z_i\setminus \{x_n\}] ,f_j\in K[Z_j\setminus \{x_n\}]$.  It follows that $ux_n^a\in u_iK[Z_i]$ and $ ux_n^a\in
u_jK[Z_j] $ for some $ {a\in \NN} $ sufficiently large. Hence$$
ux_n^a\in u_iK[Z_i]\cap u_jK[Z_j], $$ that is  a contradiction.

Let $ u\in T\setminus  \varphi(I) $ be a monomial. We claim that
there exists $i\in [r] $ such that $u\in\varphi(u_i)K[Z_i\setminus
\{x_n\}]$. Note that $\varphi (u)=u$ and $u\in I^c$ because
otherwise $u\in \varphi (I)$, which is a contradiction. This implies
that there exist $i\in[r]$ such that $u\in u_iK[Z_i].$ Hence
$$\varphi(u)=u\in \varphi(u_i)K[Z_i\setminus \{x_n\}].$$

Remains to show that we may choose $i$ such that $x_n\in Z_i$. If
$x_n\notin Z_i$ then  there exists $j\in [r]$ such that $i \neq j$
and $ t>s =\deg_{x_n} u_i$ such that $ux_n^t\in u_jK[Z_j]$ with $x_n
\in Z_j.$ Indeed, we have   $ux_n^t=u_j g,$ where $g\in K[Z_j]$ is a
monomial. It follows that $x_n^t$ does not divide $u_j$ because
$t>s,$ so $x_n$ divides $g.$ This implies $x_n\in Z_j.$
\end{proof}
\begin{Corollary}\label{main}
\[
\ \sdepth T/\varphi(I)\geq\ \sdepth S/I-1.
\]
\end{Corollary}
\begin{proof}
In the above theorem,  let ${\mathcal D}$ be a Stanley decomposition
of $S/I$ such that $\sdepth {\mathcal D}= \sdepth S/I$ . Then we
have
$$\sdepth T/\varphi(I)\geq \sdepth {\mathcal D}' = \sdepth S/I-1. $$
\end{proof}
\begin{Example}
{\em Let $I=(xy)\subset S=K[x,y]$ be an ideal, ${\mathcal
D}:S/I=xK[x]\oplus K[y]$ is a Stanley decomposition of $S/I.$ Thus
$\sdepth {\mathcal D}=1.$ After applying the map $\varphi$ defined
by $x\rightarrow 1$, ${\mathcal D}':T/\varphi(I)=K$ is a Stanley
decomposition of $T/\varphi(I)$ and $\sdepth {\mathcal D}'=0$.}
\end{Example}
\begin{Example} {\em
Let $I=(x^2,xy)$ be an ideal of  $S=K[x,y].$ A Stanley decomposition
of $S/I$ is ${\mathcal D}:S/I =  xK\oplus K[y].$ Thus for $\varphi$
given by $y\rightarrow 1$,
 ${\mathcal D}':T/\varphi(I)=K $ is a Stanley decomposition of $T/\varphi(I).$
  Here $\sdepth S/I =0$ and $\sdepth T/\varphi(I)=0$.}
\end{Example}
\begin{Example} {\em
 Let $I=(xyz)\subset S=K[x,y,z]$ be an ideal. Then  ${\mathcal D}\: S/I=K[x,z]\oplus yK[x,y]\oplus
zyK[y,z]$ is a Stanley decomposition of $S/I$ with $\sdepth
{\mathcal D}=2.$ After applying the map $\varphi$  given by
$z\rightarrow 1$, ${\mathcal D}'\: T/\varphi(I)=K[x]\oplus yK[y]$ is
a Stanley decomposition of $T/\varphi(I)$ and $\sdepth  {\mathcal
D}'=1$.}
\end{Example}

The following example shows that the inequality in Corollary~\ref{main} may be strict.
\begin{Example}
{\em Let $I=(xy,xz,xw)\subset S=K[x,y,z,w]$ be the squarefree
monomial ideal. Then $$ S/I=xK[x] \oplus K[y,z]\oplus wK[y,z,w]$$ is a Stanley decomposition of $S/I$. Thus $\sdepth S/I\geq 1$. By using partitions of the characteristic poset of $S/I$ (see \cite{HSY}), one can show that indeed $\sdepth S/I=1$.
After applying $\varphi$ we get
$\varphi(I)=(x)\subset K[x,y,z]$ and $T/\varphi(I)=K[x,y,z]/(x)\cong
K[y,z].$ Hence $\sdepth T/\varphi(I)=2.$ So we get $$\sdepth
T/\varphi(I)> \sdepth S/I.$$}
\end{Example}

We conclude this section by interpreting the inequality in Corollary~\ref{main} for squarefree monomial ideals in terms of simplicial complexes.

\medskip
Let $S=K[x_1,\ldots, x_n]$ be the polynomial ring in $n$ variables
over the field $K$ and $I \subset S $ an ideal  generated by squarefree monomials. Let $\triangle$ be  a simplicial complex on he vertex set  $[n]$
such that $I$ is the Stanley-Reisner ideal $I_{\Delta}$ associated
to $\Delta$ and $K[\Delta]=S/I.$ As above consider $T/\varphi(I).$
\begin{Lemma}\label{link}
$T/\varphi(I)=K[\link_\Delta(\{n\})].$
\begin{proof}
It is enough to show that
$\varphi(I_\Delta)=I_{\link_\Delta(\{n\})}.$ Let $G\subset[n-1]$ be
such that $x^G \in I_{\link_\Delta(\{n\})}.$ This implies that $G
\not \in \link_\Delta(\{n\})$ and so $G\cup\{n\}\not\in \Delta.$
Hence $x^{G\cup \{n\}} \in I_\Delta.$ This implies that $x^G \in
\varphi(I_\Delta).$

A square free monomial of $I_\Delta$ has the form $x^H$ with
$H\subset [n]$ and $H\not\in \Delta.$ If  $n\not\in H$ then
$x^H=\varphi(x^H)\in \varphi(I_\Delta).$ Since $H\not\in \Delta,$ we
get that $H\cup\{n\}\not\in \Delta.$ Then $H \not\in
\link_\Delta(\{n\})$ and so $x^H \in I_{\link_\Delta(\{n\})}.$ If
$n\in H$ then $x^{H\setminus\{n\}}=\varphi(x^H) \in
\varphi(I_\Delta).$ As $({H\setminus\{n\}})\cup \{n\}=H \not\in
\Delta$ we get ${H\setminus\{n\}} \not\in \link_\Delta(\{n\}).$ Thus
$x^{H\setminus\{n\}} \in I_{\link_\Delta(\{n\})}.$
\end{proof}
\end{Lemma}
\begin{Corollary}
$$\sdepth K[\link_\Delta(\{n\})]\geq \sdepth K[\Delta]-1.$$
\begin{proof}
The result follows from the above lemma and Corollary 3.2.
\end{proof}
\end{Corollary}
\begin{Corollary}
For any subset $F\subset[n],$
$$\sdepth K[\link_\Delta(F)]\geq \sdepth K[\Delta]-|F|.$$
\begin{proof}
We may assume that $n\in F.$ Apply induction on $|F|,$ the case
$|F|=1$ was done in the previous corollary. Suppose $|F|>1.$ Then by
the same corollary we get $\sdepth (K[\link_\Delta(\{n\})])\geq
\sdepth (K[\Delta])-1.$ Apply induction hypothesis for
$\link_\Delta(\{n\})$ and $F'=F\setminus\{n\}.$ Then

\begin{eqnarray*}
  \sdepth K[\link_\Delta(F)] &=&\sdepth K[\link_{\link_\Delta(\{n\})}(F')] \\
   &\geq&  \sdepth K[\link_\Delta(\{n\})]- |F'| \\
   &\geq & (\sdepth  K[\Delta]- 1)-|F'| \\
  &=& \sdepth K[\Delta] -|F|.
\end{eqnarray*}
\end{proof}
\end{Corollary}

\end{document}